\documentclass[12pt, leqno]{amsart}
\usepackage[mathcal]{euscript}
\usepackage{amssymb, amsfonts, xy, eufrak}
\usepackage{graphicx}

\xyoption{all} \CompileMatrices

\begin{document}

\def\max{\mathop{\rm max}}
\def\cf{\mathop{\rm cf}}
\def\inc{\mathop{\rm inc}\nolimits}
\def\sup{\mathop{\rm sup}\nolimits}
\def\inf{\mathop{\rm inf}\nolimits}
\def\colim{\mathop{\rm colim}}
\def\ker{\mathop{\rm ker}\nolimits}
\def\precdot{\mathop{\prec\!\!\!\cdot}\nolimits}
\def\Fun{\mathop{\rm Fun}}
\def\Map{\mathop{\rm Map}}
\def\Maps{\mathop{\rm Maps}}
\def\Hom{\mathop{\rm Hom}}
\def\Aut{\mathop{\rm Aut}}
\def\support{\mathop{\textsf{support}}}
\def\norm{\mathop{\rm norm}}
\def\CTop{\mathop{\rm CTop}}

\newtheorem{thm}{Theorem}
\newtheorem{theorem}[thm]{Theorem}
\newtheorem{lemma}[thm]{Lemma}
\newtheorem{corollary}[thm]{Corollary}
\newtheorem{proposition}[thm]{Proposition}
\newtheorem{example}[thm]{Example}

\theoremstyle{remark}
\newtheorem{rem}[thm]{Remark}
\newtheorem{definition}[thm]{Definition}

\makeatletter
\let\c@equation\c@thm
\makeatother

\newcommand{\comment}[1]{}
\newcommand{\case}[1]{{\bf \par \noindent Case #1:}}
\newcommand{\mylabel}[1]{\label{#1}}


\title{Chains of group localizations}
\author{Adam J. Prze\'zdziecki$^1$}
\address{Warsaw University of Life Sciences - SGGW, Warsaw, Poland}
\email{adamp@mimuw.edu.pl}

\maketitle
\begin{center}
\today
\end{center}

\footnotetext[1]{The author was partially supported by grant
  N N201 387034 of the Polish Ministry of Science and Higher Education.}

\begin{abstract}
  We construct long sequences of localization functors $L_\alpha$ in the
  category of abelian groups such that $L_\alpha\geq L_\beta$ for infinite cardinals $\alpha<\beta$ less than some $\kappa$.
  For sufficiently large free abelian groups $F$ and $\alpha<\beta$ we have
  proper inclusions $L_\alpha F\subsetneq L_\beta F$.

  \vspace{4pt}
  {\noindent\leavevmode\hbox {\it MSC:\ }}
  {\bf 20K40}

\end{abstract}

We reveal deeper categorical consequences of the proof of \cite[Theorem
2.1]{dugas-localizations} than those stated in the original paper. We show that:
\begin{itemize}
  \item[($\circ$)] There exists a sequence of localization functors
  $L_\lambda:{\mathcal Ab}\to{\mathcal Ab}$ in the category of
  abelian groups, indexed by infinite cardinals $\lambda$ less than
  some nonmeasurable cardinal $\kappa$, such that if
      $F$ is a free abelian group of rank at least $\kappa$
      then for $\alpha<\beta$ we have a proper inclusion
      $L_\alpha F\subsetneq L_\beta F$ which is a localization.
      More, we have $L_\alpha\geq L_\beta$ for $\alpha<\beta$ and localizations of the integers $R=L_\lambda\mathbb{Z}$
      do not depend on $\lambda$.
\end{itemize}
Constructions of this kind have been investigated before. Consider the following
sentence:
\begin{itemize}
  \item[($\ast$)] There exists a sequence of localization functors
      $L_\lambda:\mathcal{C}\to\mathcal{C}$ in a category $\mathcal{C}$
      and an object $F$ in $\mathcal{C}$ such that for $\alpha<\beta$
      we have a proper inclusion
      $L_\alpha F\subsetneq L_\beta F$ which is a localization.
\end{itemize}

The statement ($\ast$) holds in the category of graphs for $\lambda$ ranging over
cardinals less than any $\kappa$ since the ordered set $[0,\kappa)$, considered as
a category, fully embeds into the category of graphs. The validity of ($\ast$) for
$\lambda$ ranging over all cardinals is equivalent to the negation of Vop\v enka's
principle -- see \cite[Lemma 6.3]{adamek-rosicky}. In
\cite{przezdziecki} one constructs a functor from the category of graphs to the
category of groups which preserves orthogonality between morphisms and objects
(see definitions below) -- this implies that our remarks on ($\ast$) hold in the
category of groups. Existence of an analogous functor into the category of abelian
groups (which is conjectured in \cite{przezdziecki}) would translate the above to
the category of abelian groups.

We work in the category of abelian groups ${\mathcal Ab}$, although many
definitions and properties hold in more general categories (see
\cite{casacuberta-survey}). {\em Localization} is a
functor $L:{\mathcal Ab}\to{\mathcal Ab}$ with a natural transformation
$a:Id\to{\mathcal Ab}$ such that for every $X\in{\mathcal Ab}$ we have
$a_{LX}=La_X$ and $a_{LX}:LX\to LLX$ is an isomorphism. If $a_X$ is an
isomorphism then $X$ is
called {\em $L$-local}; if $Lf$ is an isomorphism then $f$ is called an
{\em $L$-equivalence}.

A homomorphism $f:X\to Y$ is {\em orthogonal} to $B$ (we write $f\perp B$) if $f$
induces, via composition, a bijection $f^*:\Hom(Y,B)\to\Hom(X,B)$. If $f:X\to Y$ is
an $L$-equivalence and $B$ is $L$-local then $f\perp B$. Conversely, if $f\perp B$
for all $L$-local $B$ then $f$ is an $L$-equivalence, and if $f\perp B$ for all
$L$-equivalences $f$ then $B$ is $L$-local. This implies that the class of
$L$-local groups is closed under limits and retracts, and the class of
$L$-equivalences is closed under colimits -- see \cite[Proposition
1.3]{casacuberta-survey}.

For any homomorphism $f:A\to B$ there exists a localization $L_f$, called an {\em
$f$-localization}, such that the class of $L_f$-local groups is
$\mathcal{D}=f^\perp=\{D\mid f\perp D\}$, and (it follows that) the class of
$L_f$-equivalences is $\mathcal{E}=\mathcal{D}^\perp=\{g:X\to Y\mid g\perp D
\mbox{ for every } D\in\mathcal{D}\}$. If $f\perp B$ then $a_A=f$ and $B=L_fA$, and
it is customary to call such a homomorphism $f$ a {\em localization}.

For any group $B$ there exists a localization functor $L_B$, called a {\em
localization at $B$}, such that the class of $L_B$-equivalences is
$\mathcal{E}=B^\perp=\{g:X\to Y\mid g\perp B\}$ and the class of $L_B$-local groups
is $\mathcal{D}=\mathcal{E}^\perp$. The existence of $f$-localizations and
localizations at a group is proved in \cite[Theorem 1]{casacuberta-large-cardinal}.

The class of localizations admits a partial ordering. We say that $L_1\geq L_2$
if one of the following, equivalent conditions holds:
\begin{enumerate}
  \item $L_2$ factors (uniquely) through $L_1$.
  \item $L_2=L_2L_1$.
  \item The class of $L_1$-local groups contains the class of $L_2$-local groups.
  \item The class of $L_2$-equivalences contains the class of $L_1$-equivalences.
\end{enumerate}
An $f$-localization is the largest localization
among those $L$ for which $f$ is an $L$-equivalence, while localization at $B$ is
the least one among those $L$ for which $B$ is $L$-local.

If $\kappa\geq\lambda$ are infinite cardinals then by $D^\kappa_{<\lambda}$ we denote the
subgroup of $\prod_\kappa D$ consisting of those functions whose support is less
than $\lambda$.

\begin{lemma}
\mylabel{lemma-support-local}
  Fix an infinite cardinal $\lambda$. If $D^\kappa_{<\lambda}$ is $L$-local for some
  $\kappa\geq\lambda$ then $D^\alpha_{<\lambda}$ is $L$-local for
  all $\alpha\geq\lambda$.
\end{lemma}
\begin{proof}
  $D^\lambda_{<\lambda}$ is a retract of $D^\kappa_{<\lambda}$, hence it is
  $L$-local.
  Let $\alpha\geq\lambda$. Each $X\subseteq\alpha$ of cardinality $\lambda$ induces a projection
  $\prod_\alpha D\to\prod_X D$. Denoting its image by $D_X$ we obtain
  $D^\alpha_{<\lambda}\to D_X\cong D^\lambda_{<\lambda}$. Then
  $D^\alpha_{<\lambda}=\lim_{{X\subseteq\alpha}\atop{|X|=\lambda}}D_X$
  is $L$-local as a limit of $L$-local groups.
\end{proof}

\begin{corollary}
\mylabel{corollary-direct-sum}
  If $\,\,S=\bigoplus_\kappa D$ is $L$-local for some infinite
  $\kappa$ then it is $L$-local for all $\kappa$.
\end{corollary}

\begin{lemma}
\mylabel{lemma-s-f-local}
  Let $f:A\to B$ be a homomorphism and $\kappa$ be an infinite regular
  cardinal greater than the number of generators of $A$. If $D$ is $L_f$-local
  then $D^{\kappa}_{<\kappa}$ is $L_f$-local.
\end{lemma}

\begin{proof}
  A homomorphism $g:A\to D^{\kappa}_{<\kappa}$ uniquely factors as
  $A\stackrel{f}{\longrightarrow}B\to\prod_\kappa D$, since the product is
  $L_f$-local. The union of the supports of all elements in $g(A)$
  forms a set $X$ whose cardinality is less than $\kappa$; hence $g(A)$
  is contained in a subgroup of $D^{\kappa}_{<\kappa}$ isomorphic to
  $\prod_XD$, hence $L_f$-local, and therefore $g$ uniquely factors through $f$.
\end{proof}

\comment{
A group $H$ is called an {\em E-ring} if it has an element $e\neq 0$ such that the
evaluation on $e$ induces an isomorphism
$\Hom(H,H)\stackrel{\cong}{\longrightarrow} H$. Thus the choice of $e$ uniquely
determines a ring structure on $H$ such that $e$ is the identity. Directly from
the definition we see that $H$ is an E-ring if and only if the homomorphism
$\mathbb{Z}\to H$ which sends $1$ to $e$ is a localization.

\begin{lemma}[{\cite[Lemma 2.3]{dugas-localizations}}]
\mylabel{lemma-equal-cardinals}
  Let $\kappa,\lambda>0$ be cardinals and $H$ an E-ring. If $F_\kappa$ is an
  abelian group of rank $\kappa$ and $f:F_\kappa\to\bigoplus_\lambda H$ is a
  localization then $\kappa=\lambda$.
\end{lemma}
\begin{proof}
  If $\kappa<\lambda$ then the image of $F_\kappa$ is contained in a proper
  summand of $\bigoplus_\lambda H$ and therefore $f$ can not be a localization
  since $\bigoplus_\lambda H$ has a nonidentity endomorphism which fixes the image of $F_\kappa$.

  If $\kappa>\lambda$ then
  $\bigoplus_\kappa F_\kappa\to\bigoplus_\kappa\bigoplus_\lambda H$
  is an $L_f$-equivalence as a colimit of $L_f$-equivalences, its target is
  $L_f$-local by Corollary \ref{corollary-direct-sum}, and therefore it is a localization.
  Uniqueness of localizations implies $\bigoplus_\lambda H\cong\bigoplus_\kappa H$.

  Let $e_0$ be a member of a basis of $F_\kappa$. We obtain a diagram
  \vspace{6pt}
  $$
  \xymatrix{
    F_\kappa \ar[d]_f \ar[r]^\pi
      & {\langle e_0\rangle} \ar[d]^i \ar@{-->}@/_1.3pc/[l] \\
    {\bigoplus_\lambda H} \ar[r]^g
      & H \ar@{-->}@/_1.3pc/[l]_r
  }
  $$
  where $\pi$ is the projection, $i(e_0)=e$, the unique map $g$ exists since
  $H$ is $L_f$ local, as a retract of $\bigoplus_\lambda H$, and $f$ is an
  $L_f$-equivalence. Since $H$ is an E-ring there exists a unique homomorphism
  $r$ such that $r(e)=f(e_0)$. Thus the homomorphism $i$ is a retract of
  $f$ hence it is an $L_f$-localization.

  Then the composition
  $F_\lambda\stackrel{\bigoplus_\lambda i}{\longrightarrow}\bigoplus_\lambda
  H\cong \bigoplus_\kappa H$ is a localization, which was excluded at
  the beginning of this proof.
\end{proof}
}

Let $L$ be a localization. We look at the composition
$$F_\kappa=\bigoplus_\kappa\mathbb{Z}\stackrel{\bigoplus_\kappa
a_\mathbb{Z}}{\longrightarrow}\bigoplus_\kappa L\mathbb{Z}
\subseteq\prod_\kappa L\mathbb{Z}.$$
Since the product is $L$-local, it factors as
\begin{equation}
\mylabel{equation-g}
F_\kappa\stackrel{a}{\longrightarrow} LF_\kappa\stackrel{g}{\longrightarrow}\prod_\kappa L\mathbb{Z}
\end{equation}
where $a=a_{F_\kappa}$.

\rem\mylabel{remark-local-image} Let $N^\kappa_L$ denote the image of $g$.
Since $F_\kappa$ is a free group, it is easy to
see that $N^\kappa_L$ is $L_a$-local.
In fact, $N^\kappa_L$ may be described as the least $L_a$-local subgroup of
$\prod_\kappa L\mathbb{Z}$ which contains $\bigoplus_\kappa\mathbb{Z}$.

\definition\mylabel{definition-support-k}
Define $\support_\kappa L$ as the least cardinal greater than the cardinalities of
the supports of all elements in $N^\kappa_L$.

\rem\mylabel{remark-support-independent-on-k} The number $\support_\kappa L$ does
not depend on the choice of basis for $F_\kappa$: if $B$ and $C$ are two such
bases then a bijection $\alpha:B\to C$ induces a diagram
$$
\xymatrix{
  {\bigoplus_{b\in B}\mathbb{Z}} \ar[r] \ar[d]
    & LF_\kappa \ar[r]^g \ar[d]
    & {\prod_{b\in B}L\mathbb{Z}} \ar[d] \\
  {\bigoplus_{c\in C}\mathbb{Z}} \ar[r]
    & LF_\kappa \ar[r]^{g'}
    & {\prod_{c\in C}L\mathbb{Z}}
}
$$
where the rightmost vertical arrow permutes the components preserving supports of
elements.

\definition\mylabel{definition-support-l}
Define $\support L$ to be the supremum of $\support_\kappa L$ over all cardinals
$\kappa$, or $\infty$ if this class of cardinals is unbounded.

An embedding of a subset $X\subseteq\kappa$ induces a diagram
$$
\xymatrix{
  {\bigoplus_\kappa\mathbb{Z}} \ar[r] \ar[d]
    & LF_\kappa \ar[r]^g \ar[d]
    & {\prod_{\kappa}L\mathbb{Z}} \ar[d] \\
  {\bigoplus_{X}\mathbb{Z}} \ar[r]
    & LF_X \ar[r]^{g'}
    & {\prod_{X}L\mathbb{Z}}
}
$$
where the vertical arrows are retractions. This allows comparing possible
cardinalities of supports of elements of
$N^\kappa_L\subseteq\prod_\kappa L\mathbb{Z}$ for different
$\kappa$'s, and therefore it proves:
\begin{lemma}
\mylabel{lemma-support}
  If $\support_\kappa L\leq\kappa$ then $\support L=\support_\kappa L$.
\end{lemma}

\begin{lemma}
\mylabel{lemma-standard-localizations}
  Let $L$ be a localization. The following are equivalent:
  \begin{enumerate}
    \item $\support L=\omega_0$.
    \item $LF_{\omega_0}=\bigoplus_{\omega_0}L\mathbb{Z}$.
    \item For any $\kappa$ we have $LF_\kappa=\bigoplus_\kappa L\mathbb{Z}$.
  \end{enumerate}
\end{lemma}

\begin{proof}
  (3)$\implies$(1) and (3)$\implies$(2) are obvious; (2)$\implies$(3)
  follows from Corollary \ref{corollary-direct-sum}.
  It remains to prove (1)$\implies$(3). If $\support L=\omega_0$
  then we have an epimorphism
  $g:LF_\kappa\to N^\kappa_L\cong\bigoplus_\kappa L\mathbb{Z}$.
  Since $LF_\kappa$ is $L$-local and the target of $g$ is
  $L$-equivalent to the free group $F_\kappa$ via an $L$-equivalence
  $\bigoplus_\kappa(\mathbb{Z}\to L\mathbb{Z})$ we see that
  $g$ has a right inverse $r$. Then $r(N^\kappa_L)$ is a retract of
  $LF_\kappa$ which contains $F_\kappa$, thus $r$ is onto and $g$
  is an isomorphism as claimed.
\end{proof}

A localization satisfying the conditions of Lemma
\ref{lemma-standard-localizations} is called in
\cite{dugas-localizations} a {\em standard localization}.

\begin{lemma}
\mylabel{lemma-large-kappa}
  Let $\kappa$ be an infinite cardinal less than the first measurable cardinal.
  Then there exists a localization $L$ such that $\support L>\kappa$.
\end{lemma}

\begin{proof}
  At the heart of the proof of \cite[Theorem 2.1]{dugas-localizations}
  lies a construction of a localization homomorphism
  $\varepsilon:F_\kappa\to M$ such that for a certain group $R$ we have
  $\bigoplus_\kappa R\subseteq M\subseteq\prod_\kappa R$ and
  $M$ contains functions which are nowhere zero and
  $R=L_\varepsilon\mathbb{Z}$. This implies our claim.
\end{proof}

\begin{theorem}
\mylabel{theorem-chain}
  Let $\kappa$ be an infinite cardinal less than the first measurable cardinal.
  There exists a sequence of localization functors $L_\alpha$ for $\alpha<\kappa$, such that:
  \begin{itemize}
    \item[(a)] $\support L_\alpha=\alpha^+$,
    \item[(b)] $L_\alpha\geq L_\beta$ for $\alpha<\beta<\kappa$,
    \item[(c)] $L_\alpha F_\kappa\subsetneq L_\beta F_\kappa$ for $\alpha<\beta<\kappa$,
  \end{itemize}
  where $\alpha^+$ is the successor cardinal of $\alpha$.
\end{theorem}

\begin{proof}
  Let $L$ be the localization from Lemma \ref{lemma-large-kappa} and
  $f_\alpha:F_\alpha\to LF_\alpha$ be the localization homomorphism.
  Define $L_\alpha=L_{f_\alpha}$. Since $L_\alpha F_\alpha=LF_\alpha$
  is a retract of $LF_\kappa$, an argument as in the proof of Lemma
  \ref{lemma-support} implies that $\support L_\alpha>\alpha$.
  Lemma \ref{lemma-s-f-local} for $\kappa=\alpha^+$ and Lemma
  \ref{lemma-support-local} imply that $R^\kappa_{<\alpha^+}$ is
  $L_\alpha$-local for all $\kappa>\alpha$, hence Remark \ref{remark-local-image}
  implies that $\support L_\alpha\leq\alpha^+$, which yields (a). Since for
  $\alpha<\beta$ the map $f_\alpha$ is a retract of $f_\beta$, items (b) and (c) follow easily.
\end{proof}

If $f:\mathbb{Z}\to R=L_\varepsilon\mathbb{Z}$ is an $L_\varepsilon$
localization of $\mathbb{Z}$ as in the proof of Lemma \ref{lemma-large-kappa}
then the $f$-localization $L_f$ is strictly greater, while the localization at $R$,
$L_R$, is strictly less than all the localizations $L_\alpha$. We do not
know if $L=L_R$; it is still conceivable that $\support L_R$ might exceed $\kappa^+$.

In the proof of Lemma \ref{lemma-large-kappa} the groups $R$ and $M=LF_\kappa$
have the same cardinality $\lambda\geq 2^\kappa$, hence also the groups
$L_\alpha F_\kappa$ have cardinality $\lambda$ each. This cannot happen
if we want $\alpha$ to run over all cardinals, as we speculated in the introduction.

In principle, one could construct similar sequences of localizations based on the structure of
the kernels of maps $g$ in Diagram (\ref{equation-g}), but we are unaware
of any examples of nontrivial kernels of $g$. Dugas and Feigelstock prove in
\cite[Theorem 1.8]{dugas-self-free} that in certain cases these kernels must be trivial.

\end{document}